\documentclass{amsart}

\usepackage{stylesheet}

\title{Tropicalizing the Graph Profile of Some Almost-Stars}

\date{}

\author{Maria Dasc\u{a}lu}
\address{Department of Mathematics and Statistics,
Lederle Graduate Research Tower, 1623D,
University of Massachusetts Amherst,
710 N. Pleasant Street,
Amherst, MA 01003} \email{mdascalu@umass.edu}

\author{Annie Raymond}
\address{Department of Mathematics and Statistics,
Lederle Graduate Research Tower, 1623D,
University of Massachusetts Amherst,
710 N. Pleasant Street,
Amherst, MA 01003} \email{raymond@math.umass.edu}

\thanks{Maria Dasc\u{a}lu and Annie Raymond were partially supported by NSF grant DMS-2054404.}

\begin{document}

\maketitle

\begin{abstract}
    Many important problems in extremal combinatorics can be stated as certifying polynomial inequalities in graph homomorphism numbers, and in particular, many ask to certify pure binomial inequalities. For a fixed collection of graphs $\U$, the \emph{tropicalization of the graph profile of $\U$} essentially records all valid pure binomial inequalities involving graph homomorphism numbers for graphs in $\U$. Building upon ideas and techniques described by Blekherman and Raymond in 2022, we compute the tropicalization of the graph profile for $K_1$ and $S_{2,1^k}$-trees, almost-star graphs with one branch containing two edges and $k$ branches containing one edge. This allows pure binomial inequalities in homomorphism numbers (or densities) for these graphs to be verified through an explicit linear program where the number of variables is equal to the number of edges in the biggest $S_{2,1^k}$-tree involved.
\end{abstract}

\section{Introduction}

The \emph{number of homomorphisms} from a graph $H$ to a graph $G$, denoted by $\hom(H;G)$, is the number of maps from $V(H)$ to $V(G)$ that yield a graph homomorphism, i.e., that map every edge of $H$ to an edge of $G$. Many important problems and results in extremal graph theory can be framed as certifying the validity of polynomial inequalities in the number of graph homomorphisms which are valid on all graphs. In particular, in many cases, those polynomials are \emph{pure binomial inequalities}. For example, Sidorenko's conjecture \cite{Sid93} can be stated as $\hom(P_0;G)^{2|E(H)|-|V(H)|} \cdot \hom(H;G) \geq  \hom(P_1;G)^{|E(H)|}$ for any bipartite graph $H$, where $P_i$ denotes the path with $i$ edges. The Erd\H os-Simonovits Theorem \cite{erdossimonovits}, recently proven in \cite{saglam} and then generalized in \cite{BR}, states that $\hom(P_{2u};G)^2 \cdot \hom(P_{2v+1};G)^{2v-1-2u} \geq \hom(P_{2v-1}; G)^{2v+1-2u}$ for $u<v$ (the original conjecture was for $u=0$). One last example is the Kruskal-Katona Theorem \cite{Kruskal, Katona} which states that $\hom(K_p; G)^q \geq \hom(K_q; G)^p$ for $2\leq p < q$, where $K_m$ is the complete graph on $m$ vertices. 

\emph{Graph profiles} essentially record all valid polynomial inequalities in homomorphism numbers of some given graphs. In \cite{BRST2, BR}, it was shown that the \emph{tropicalization} of graph profiles is a way of capturing all true pure binomial inequalities. Graph profiles are complicated objects---for instance, they need not be semialgebraic. Very few two-dimensional density profiles are known, e.g., \cite{RazTriangle, Nikiforov, MR3549620}, and none are known in higher dimensions. Though it still retains interesting information, the tropicalization of a graph profile is a much simpler object: it is always a closed convex cone and it has been computed for some arbitrarily large families of graphs \cite{BRST2,BR}. Moreover, all tropicalizations computed thus far have been rational polyhedral cones. 

In this paper, for $m\in \NN$, we explicitly compute the rational polyhedral cone forming the tropicalization of the graph profile for the collection of graphs $\mathcal{U}=\{S_0, S_{2,1^0}, S_{2,1^1}, \ldots, S_{2,1^{m-1}}\}$, where $S_0$ is a single vertex and $S_{2,1^k}$ is the tree where $V(S_{2, 1^k}) = \{1, 2, \ldots, k+3\}$ and $E(S_{2,1^k}) = \{\{1, j\} \mid j \in \{2, 3, \ldots, k+2\}\} \cup \{\{k+2, k+3\}\}$, i.e., a star where one branch is elongated to have length two. The notation comes from thinking of the branch lengths of this almost-star as a partition of the number of edges in the graph. The tree $S_{2, 1^3}$ is depicted in Figure \ref{fig:s213tree}. As a consequence, the validity of any pure binomial inequality in the graphs of $\U$ can be checked in a finite way from an explicit finite collection of binomial inequalities, and this can be done through a linear program.

\begin{figure}
    \centering
    \includegraphics[scale=.5]{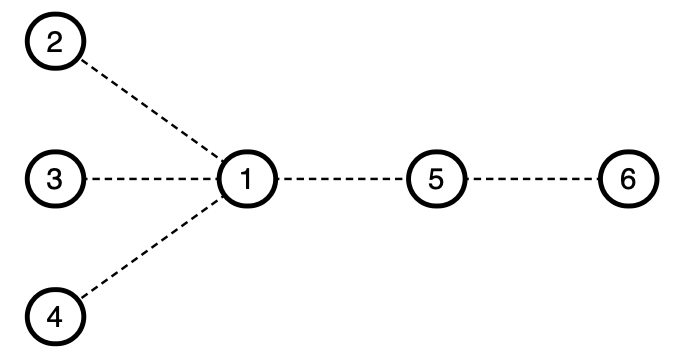}
    \caption{The $S_{2,1^3}$ tree.}
    \label{fig:s213tree}
\end{figure}

We now make these ideas precise.

\subsection{Definitions, Results and Proof Strategies:}

 The {\em number graph profile} of a collection of connected graphs $\mathcal{U} = \{H_1, \ldots, H_s \}$, denoted as $\mathcal{N}_\mathcal{U}$,
 is the set of all vectors $(\hom(H_1;G), \hom(H_2;G), \ldots, \hom(H_s;G))$ as $G$ varies over all graphs. For any $\mathcal{U}$, the number graph profile $\NU$ is contained in $\mathbb{N}^s$. Understanding all $s$-tuples that can occur as homomorphism numbers in $\mathcal{U}$ is essentially equivalent to understanding all polynomial inequalities in homomorphism numbers which are valid on all graphs. It is known that the problem of checking whether a polynomial expression in homomorphism numbers is nonnegative on all graphs is undecidable \cite{ioannidisramakrishnan}. 
 
\emph{A pure binomial inequality} has the form $\mathbf{x}^{\bm{\alpha}} \geq \mathbf{x}^{\bm{\beta}}$ where $\mathbf{x} \in \mathbb{R}^s_{\geq 0}$ with $\bm{\alpha}, \bm{\beta} \in \mathbb{R}_{\geq 0}^s$. As mentioned in the introduction, pure binomial inequalities in homomorphism numbers are of interest in extremal graph theory. To study such inequalities, let $\log \,:\, \RR^s_{>0} \rightarrow \RR^s$ be defined as $\log({\bf v}) := (\log(v_1), \ldots, \log (v_s))$. 
For a set $\mathcal{S} \subseteq \RR^s_{\geq 0}$, we define $\log(\mathcal{S}) := \log (\mathcal{S} \cap \RR^s_{>0})$.
The tropicalization of $\mathcal{S}$, which is also called the {\em logarithmic limit set} of $\mathcal{S}$, is defined to be
$$\trop(\mathcal{S}) := \lim_{\tau \rightarrow \infty} \log_{\tau}(\mathcal{S}).$$ In \cite[Lemma 2.2]{BRST2}, it was shown that if $\mathcal{S}\subseteq \mathbb{R}^s_{\geq 0}$ has the Hadamard property, then $\trop(\mathcal{S})$ is a closed convex cone. Moreover, the tropicalization of $\mathcal{S}$ is equal to the closure of the conical hull of $\log(\mathcal{S})$, and so $\trop(\NU) = \textup{cl}(\textup{cone}(\log(\NU)))$. Each pure binomial inequality in numbers corresponds to \emph{a linear inequality in logarithms:} $\langle \bm{\alpha}, \log \mathbf{x} \rangle \geq \langle \bm{\beta}, \log \mathbf{x} \rangle$. The extreme rays of the dual cone $\tropNU^*$ generate all of the pure binomial inequalities valid on $\NU$. Note that $\trop(\NU) \subseteq \mathbb{R}_{\geq 0}^s$ since $\NU \cap \RR^s_{>0}$ contains only points where every coordinate is at least one. In \cite[Proposition 2.4]{BR}, it was shown that no spurious binomial inequalities are added by removing points with zero coordinates from $\NU$. 

Blekherman and Raymond conjectured that $\tropNU$ is in fact a rational polyhedral cone for any finite collection of graphs $\mathcal{U}$ \cite[Conjecture 2.14]{BR}. They proved that this is the case when $\U$ consists of complete graphs, even cycles, odd cycles or paths by fully characterizing $\tropNU$ in those settings. They also showed that the conjecture holds for the following collections of graphs.

\begin{theorem}{\cite[Theorems 3.2 and 1.2]{BR}} \label{thm:informalchordalseriesparallel}
Let $\mathcal{U}$ be a finite collection of chordal series-parallel graphs. Then $\tropNU$ is a
rational polyhedral cone. Furthermore, there exists a finite collection of binomial inequalities such that any pure binomial inequality in the graphs of $\U$ can be deduced in a finite way from this finite collection.
\end{theorem}

Note that this theorem already implies that $\tropNU$ is a rational polyhedral cone for the collection of graphs we are interested in, namely $\U= \{S_0, S_{2,1^0}, S_{2,1^1}, \ldots, S_{2,1^{m-1}}\}$, since all of these graphs are chordal and series-parallel. The main contribution of this paper is finding an explicit description for $\tropNU$.

\begin{restatable*}[]{theorem}{maintheorem}
\label{thm:trop}
Let $\U = \{S_0, S_{2, 1^0}, S_{2, 1^1}, \ldots, S_{2, 1^{m-1}}\}$ and let

$$
Q = \left\{\yy \in \RR^{m+1} \middle\vert \begin{array}{l}
- y_1 + y_2  \geq 0 \\
4y_1 - 3y_2 \geq 0 \\
3y_1 -3y_3 + y_4 \geq 0 \\
y_1 + 2y_{m-1} -2y_m \geq 0 \\
y_0 + y_{m-1} - y_m \geq 0 \\
y_0 - 2y_1 + y_3 \geq 0 \\
y_{i-1} - 2y_i + y_{i+1} \geq 0 \quad \forall 2\leq i \leq m-1 \\
m\cdot y_{m-1} - (m-1)\cdot y_m \geq 0 
\end{array}\right\}.
$$ Then $\tropNU=Q$. 
\end{restatable*}

Having such an explicit description of $\tropNU$ allows one to check the validity of any pure binomial inequality in the graphs of $\U$ through an explicit linear program.  It is more convenient to check the validity of a binomial inequality for $\NU$ by checking the validity of the corresponding linear inequality for $\tropNU$. There, we simply need to show that this linear inequality can be written as a conical combination of the defining inequalities for $\tropNU$. If there is no such conical combination, then the original pure binomial inequality is not valid for $\NU$. Finding such a conical combination can be done via a linear program. Indeed, suppose one wants to check whether $S_0^{\alpha_0}\prod_{i=0}^{m-1} S_{2,1^i}^{\alpha_{i+1}} \geq S_0^{\beta_0}\prod_{i=0}^{m-1} S_{2,1^i}^{\beta_{i+1}}$ is a valid inequality for some $\alpha_i,\beta_i$ (some of which can be 0). This is equivalent to checking that $\alpha_0 y_0 + \sum_{i=1}^{m} \alpha_{i} y_i - \beta_0 y_0 - \sum_{i=1}^{m} \beta_{i} y_i \geq 0$ on $\tropNU$ (since $y_0 = \log(\hom(S_0;G))$ and $y_i=\log(\hom(S_{2,1^{i-1}};G))$ for $1 \leq i \leq m$). Thus one can simply minimize $\alpha_0 y_0 + \sum_{i=1}^{m} \alpha_{i} y_i - \beta_0 y_0 - \sum_{i=1}^{m} \beta_{i} y_i$ over the cone $Q$ of Theorem~\ref{thm:trop}. If the optimal value is 0, then the inequality is valid, and the dual solution gives the conical combination of inequalities of $Q$ that yields $\alpha_0 y_0 + \sum_{i=1}^{m} \alpha_{i} y_i - \beta_0 y_0 - \sum_{i=1}^{m} \beta_{i} y_i \geq 0$, i.e., how one recovers this inequality from the pure binomial inequalities associated to the inequalities in $Q$. Otherwise, the inequality is not valid. 

Note that these ideas carry on over to \emph{homomorphism densities}. The homomorphism density from a graph $H$ to a graph $G$, denoted as $t(H;G)$, is the probability that a map from $V(H)$ to $V(G)$ is a graph homomorphism, i.e., $t(H;G)=\frac{\hom(H;G)}{|V(G)|^{|V(H)|}}$. The \emph{density graph profile} $\GU$ of some finite collection of graphs $\U$  is the closure of the points $(t(H_1;G), t(H_2;G), \ldots, t(H_s;G))$ as $G$ varies over all graphs. It is more natural to write certain problems from extremal graph theory as certifying the validity of a polynomial inequality in densities. In this setting again, it is known that the problem of checking whether a polynomial expression in densities is nonnegative on all graphs is undecidable \cite{HN11}. Note that any pure binomial inequality in homomorphism densities, say $t(H_1;G)-t(H_2;G)\geq 0$ (where $H_1$ and $H_2$ need not be connected graphs), can be rewritten as a pure binomial inequality in homomorphism numbers. Indeed, $t(H_1;G)-t(H_2;G)=\frac{\hom(H_1;G)}{|V(G)|^{|V(H_1)|}} - \frac{\hom(H_2;G)}{|V(G)|^{|V(H_2)|}}$, and since $|V(G)|=\hom(K_1;G)$, the previous inequality can be rewritten as $\hom(K_1;G)^{|V(H_2)|-|V(H_1)|}\hom(H_1;G) - \hom(H_2;G)\geq 0$ if $|V(H_2)|\geq |V(H_1)|$ or $\hom(H_1;G) - \hom(K_1;G)^{|V(H_1)|-|V(H_2)|}\hom(H_2;G)\geq 0$ if $|V(H_1)|\geq |V(H_2)|$. Therefore, checking the validity of $t(H_1;G)-t(H_2;G)\geq 0$ over $\GU$ where $\U$ contains all connected components in $H_1$ and $H_2$ can be done by checking the corresponding pure binomial inequality in homomorphism numbers over $\NU$ if $\U$ contains $K_1$ or over $\mathcal{N}_{\mathcal{U}'}$ where $\U'=\U\cup \{K_1\}$ otherwise. In our setting, $\U$ contains $K_1=S_0$, so by Theorem~\ref{thm:trop}, the validity of any pure binomial inequality in densities in $\U$ can also be checked through the linear program described in the previous paragraph.

Our strategy to prove Theorem~\ref{thm:trop} is to first certify that all of the inequalities of $Q$ come from valid pure binomial inequalities on the profile $\NU$, thus showing that $\tropNU \subseteq Q$. Certification is done using some standard techniques as well as tools presented in \cite{koppartyrossman}. To show that $Q\subseteq \tropNU$,  we find all extreme rays of $Q$ and show that these rays are \emph{realizable}. We say a ray $\mathbf{r}$ is realizable if there exists a graph $G$ or a sequence of graphs $G_n$ on $n$ vertices such that  $\alpha(\log \hom(H_1;G), \ldots, \log \hom(H_s;G))=\mathbf{r}$ or $\alpha(\log \hom(H_1;G_n), \ldots, \log \hom(H_s;G_n))\rightarrow \mathbf{r}$ as $n\rightarrow \infty$ respectively for some constant $\alpha\in \mathbb{R}_{\geq 0}$. These sequences of graphs often arise by constructing blow-up graphs, and taking tensor powers and disjoint unions of those graphs. 

The reader might notice that all inequalities in $Q$ have a single negative term. This is no coincidence: Corollary 2.9 of \cite{BR} proves that each extreme ray of the dual cone $\tropNU^*$ for any finite collection of connected graphs $\U$ is spanned by a vector with at most one negative coordinate.

Another nice property of $\tropNU$ is that it is \emph{max-closed}: if $(x_1, \ldots, x_s), (y_1, \ldots, y_s)\in \tropNU$, then $(x_1, \ldots, x_s)\oplus (y_1, \ldots, y_s):=(\max\{x_1,y_1\}, \ldots, \max\{x_s,y_s\})\in \tropNU$. 
Given a cone  $\mathcal{S}\subset \mathbb{R}_{\geq 0}^s$, \emph{the double hull} of $\mathcal{S}$ is the smallest closed and max-closed convex cone containing $\mathcal{S}$.
A point $\mathbf{p}\in \mathcal{S}$ is said to be \emph{max-extreme} if whenever $\mathbf{p}=\mathbf{x}\oplus \mathbf{y}$ with $\mathbf{x},\mathbf{y} \in \mathcal{S}$, then $\mathbf{p}=\mathbf{x}$ or $\mathbf{p}=\mathbf{y}$. Furthermore, $\mathbf{p}\in \mathcal{S}$ is said to span a \emph{doubly extreme ray} of $\mathcal{S}$ if $\mathbf{p}$ spans an extreme ray of $\mathcal{S}$ and $\mathbf{p}$ is max-extreme.
Theorem 2.13 in \cite{BR} states  that a closed and max-closed convex cone in $\mathbb{R}^s_{\geq 0}$ is the double hull of its doubly extreme rays.

As a corollary of Theorem~\ref{thm:trop}, we rewrite the tropicalization $\tropNU$ as the double hull of five doubly extreme rays. 

\begin{restatable*}[]{corollary}{maincorollary}
\label{cor:dh}
The set $\tropNU$ is the double hull of the following five doubly extreme rays in $\RR^{m+1}$:  $\mathbf{d}_{1,m}:=(1, 0, 0, \ldots, 0)$, $\mathbf{d}_{2,m}:=\vec{1}$, $\mathbf{d}_{3,m}:=(1, 2, 2, 3, 4, \ldots, m)$, $\mathbf{d}_{4,m}:=(1, 3, 4, 5, \ldots, m+2)$, $\mathbf{d}_{5,m}:=(2, 4, 5, 7, \ldots, 2m+1)$.
\end{restatable*}

\subsection{Outline}
In Section~\ref{sec:inequalities}, we show that the defining inequalities of $Q$ in Theorem~\ref{thm:trop} are valid for $\tropNU$. In Section~\ref{sec:rays}, we show that the rays in Corollary~\ref{cor:dh} are present in $\tropNU$, and we use those to construct a larger family of rays in $\tropNU$. This larger family will turn out to be the extreme rays of $\tropNU$, as will be shown in the proof of Theorem~\ref{thm:trop} in Section~\ref{sec:results}. In Section~\ref{sec:background}, we briefly recall helpful properties of graph homomorphisms and of tropicalizations of graph profiles, as well as some results of Kopparty and Rossman (\cite{koppartyrossman}) which we use later on. 
\section{Background on tropicalizations, homomorphisms and Kopparty-Rossman}\label{sec:background}

We go over different concepts and results which are used to prove Theorem~\ref{thm:trop} and Corollary~\ref{cor:dh} as well as the lemmas that precede them. 

\subsection{Useful results about tropicalizations}\label{sec:tropresults}

Recall from the introduction that for vectors $\mathbf{x}, \mathbf{y} \in \mathbb{R}^s$, we let $\mathbf{x}\oplus \mathbf{y}$ denote their \emph{tropical sum}: $$\mathbf{x}\oplus \mathbf{y}=(\max\{x_1, y_1\}, \ldots, \max\{x_s, y_s\}).$$ A set $\mathcal{S}\subseteq \mathbb{R}^s$ is said to be \emph{max-closed} if for any $\mathbf{x}, \mathbf{y} \in \mathcal{S}$, we have $\mathbf{x}\oplus \mathbf{y}\in \mathcal{S}$. Corollary 2.3 in \cite{BR} proves that $\tropNU$ is a max-closed convex cone for any finite collection of connected graphs $\U$.

Recall also from the introduction that the \emph{double hull} of a cone $\mathcal{S}\subset \mathbb{R}^s_{\geq 0}$ is defined to be the smallest max-closed convex cone containing $\mathcal{S}$ and is denoted by $\operatorname{dh}(\mathcal{S})$. From Theorem 2.13 in \cite{BR}, we know that for any closed convex cone $\mathcal{S} \subseteq \RR^s_{\geq 0}$ that is also max-closed, $\mathcal{S}$ is equal to the double hull of its doubly extreme rays. In particular, this means that $\tropNU$ is equal to the double hull of its doubly extreme rays. Essential to the proof that $\tropNU$ is a max-closed convex cone (Lemma 2.2 of \cite{BR}) are the following two graph operations.  It has long been known that $\hom(H;G_1)+\hom(H;G_2)=\hom(H;G_1G_2)$ where $G_1G_2$ is the disjoint union of $G_1$ and $G_2$, and that $\hom(H;G_1)\cdot \hom(H;G_2)=\hom(H;G_1\times G_2)$ where $G_1\times G_2$ is the categorical product of $G_1$ and $G_2$. From Lemma~2.2 in \cite{BRST2}, to show that a conical combination of two points $\log(\mathbf{v}),\log(\mathbf{w}) \in \log(\NU)$ is in $\tropNU$, it suffices to show that $a_1 \log(\mathbf{v})+a_2\log(\mathbf{w})\in \tropNU$ for $a_1,a_2\in \NN$. Since $a_1 \log(\mathbf{v})+a_2\log(\mathbf{w})= \log(\mathbf{v}^{a_1} \cdot \mathbf{w}^{a_2})$, we see that this is true by using our graph operations. Similarly, from Lemma~2.2 in \cite{BR}, to show that a tropical combination of these two points is in $\tropNU$, it suffices to show that $a_1 \log(\mathbf{v})\oplus a_2\log(\mathbf{w})\in \tropNU$ for $a_1,a_2\in \NN$. It turns out that $a_1 \log(\mathbf{v})\oplus a_2\log(\mathbf{w})=\frac{\log(\mathbf{v}^{a_1 l}+\mathbf{w}^{a_2l})}{l}$ as $l\rightarrow \infty$, and so we see that the desired result holds since $\frac{\log(\mathbf{v}^{a_1 l}+\mathbf{w}^{a_2l})}{l}$ is in $\tropNU$ because of our graph operations. Therefore, taking the conical hull and the max-closure of any set of rays can also be understood in a graph theoretical way. 

\subsection{Gluing algebra}

In the introduction, we defined $\hom(H; G)$ when $H$ is an unlabeled graph to be the number of homomorphisms from $H$ to $G$. We now extend this definition for when $H$ is a partially labeled graph. This will be useful to prove inequalities in the next section. A graph is {\em partially labeled} if a subset $L$ of its vertices are labeled with elements of $\mathbb{N} := \{1,2,3,\ldots\}$ such that no vertex receives more than one label. If no
vertices of $H$ are labeled, then $H$ is {\em unlabeled}. Let $\vartheta: L \rightarrow V(G)$. Then for a partially labeled graph $H$, $\hom(H;G)$ is the number of homomorphisms $\varphi: V(H)\rightarrow V(G)$ that extend $\vartheta$. 

Furthermore, for two graphs (partially labeled or not), $\hom(H_1;G)\cdot \hom(H_2;G)=\hom(H_1 H_2;G)$ where $H_1H_2$ is the graph obtained by gluing $H_1$ and $H_2$ along vertices with the same labels, and replacing any doubled edge with a single edge. This is best illustrated by an example: $$\hom\left(\Htwo{2}{1}{3};G\right) \cdot \hom\left(\Htwo{1}{2}{};G\right) = \hom\left(\pthree{}{1}{2}{3};G\right).$$  

Note that multiplying two unlabeled graphs simply yields their disjoint union.

For a partially labeled graph $H$, $\left[\left[\hom(H; G)\right]\right]:=\sum_{\psi:L \rightarrow V(G)} \hom(H_\psi;G)$ where $H_\psi$ corresponds to the partially labeled graph where the vertices in $L$ are labeled according to $\psi$ instead of $\vartheta$. This operation yields the \emph{unlabeling} of $H$, i.e., $\left[\left[\hom(H; G)\right]\right]=\hom(H';G)$ where $H'$ is the unlabeled version of $H$. For example, $\left[\left[\hom\left(\Htwo{1}{2}{3};G\right)\right]\right] = \hom\left(\uHtwo;G\right).$

We will often use $H$ as a shorthand for $\hom(H;G)$, especially for the purposes of writing inequalities. By $H^k$, we denote both  $\hom(H;G)^k$ and $\hom(k \textup{ disjoint copies of } H; G)$ as they are equal. The previous example thus can be written as $\Htwo{2}{1}{3} \ \Htwo{1}{2}{} = \pthree{}{1}{2}{3},$ and unlabeling can be represented as $\left[\left[\Htwo{1}{2}{3}\right]\right] = \uHtwo.$

\subsection{Useful results of Kopparty and Rossman}\label{subsec:kr}

The concept of the homomorphism domination exponent was introduced in \cite{koppartyrossman}, though the idea behind it had been central to many problems in extremal graph theory for a long time. Let the \emph{homomorphism domination exponent} of a pair of graphs $F_1$ and $F_2$, denoted by $\HDE(F_1;F_2)$, be the maximum value of $c$ such that $\hom(F_1;G) \geq \hom(F_2;G)^c$ for every graph $G$. Note that this yields a pure binomial inequality. 

In \cite{koppartyrossman}, Kopparty and Rossman showed that $\HDE(F_1; F_2)$ can be found by solving a linear program in $2^{|V(F_2)|}$ variables when $F_1$ is chordal and $F_2$ is series-parallel. Since this is the case when $F_1$ and $F_2$ are unions of $S_{2,1^k}$'s, this result was very useful to help us figure out what $\tropNU$ should be. 
We now recall one of their results which we will use to prove one of our binomial inequalities in the next section.

Let $\Hom(F_1;F_2)$ be the set of homomorphisms from $F_1$ to $F_2$, and let $\mathcal{P}(F_2)$ be the polytope consisting of normalized $F_2$-polymatroidal functions:
\begin{align*}
\mathcal{P}(F_2)=\big\{p\in \mathbb{R}^{2^{|V(F_2)|}} | \,\, & p(\emptyset) = 0 &&\\
& p(V(F_2)) = 1 && \\
& p(A) \leq p(B) && \forall \  A \subseteq B \subseteq V(F_2)\\
& p(A\cap B) + p(A\cup B) \leq p(A)+p(B) && \forall \ A, B \subseteq V(F_2)\\
& p(A\cap B) + p(A\cup B) = p(A)+p(B) &&\forall \ A, B \subseteq V(F_2) \textup { such that } A\cap B \\
&&& \quad \quad \quad \textup{ separates } A\backslash B \textup{ and }B\backslash A\big\}.
\end{align*}
In the definition above, $A\cap B$ is said to separate $A\backslash B$ and $B\backslash A$ if there is no path in $F_2$ going from a vertex in $A\backslash B$ to a vertex in $B\backslash A$. 

\begin{theorem}[Kopparty-Rossman, 2011]\label{thm:kr11}
Let $F_1$ be a chordal graph and let $F_2$ be a series-parallel graph. Then

$$\HDE(F_1; F_2) = \min_{p\in \mathcal{P}(F_2)} \max_{\varphi\in \Hom(F_1;F_2)}  \sum_{S \subseteq \textup{MaxCliques}(F_1)} -(-1)^{|S|} p(\varphi(\cap S))\\$$
where $\textup{MaxCliques}(F_1)$ is the set of maximal cliques of $F_1$ and $\cap S$ is the intersection of the maximal cliques in $S$. 
\end{theorem}
\section{Valid Binomial Inequalities for $S_{2,1^k}$-trees}\label{sec:inequalities}

In this section, using some standard techniques as well as the result of Kopparty and Rossman presented in Section~\ref{sec:background}, we derive some valid pure binomial inequalities for $\NU$ that yield the inequalities defining $Q$ in Theorem~\ref{thm:trop} after taking the log, thus showing that these are valid inequalities for $\tropNU$. Recall that $V(S_{2, 1^k}) = \{1, 2, \ldots, k+3\}$, and $E(S_{2,1^k}) = \{\{1, j\} \mid j \in \{2, 3, \ldots, k+2\}\} \cup \{\{k+2, k+3\}\}$. Finally, for ease for reading, in this section, we will use $H$ as a shorthand for $\hom(H; G)$ whenever there is no confusion.

First note that $S_{2,1^1} - S_{2,1^0} \geq 0$ and $S_{2,1^0}^4 - S_{2,1^1}^3\geq 0$ are inequalities involving only paths with two and three edges, and that these inequalities are already known (see \cite{koppartyrossman} for example). The following inequality is also trivially true.
\begin{theorem}
    We have that $S_0 S_{2, 1^{m-2}} \geq S_{2, 1^{m-1}}$ is a valid inequality for all graphs $G$.
\end{theorem}
\begin{proof}
    The inequality holds since there is a surjective homomorphism $\varphi$ from $S_0 S_{2, 1^{m-2}}$ to $S_{1, 2^{m-1}}$. Indeed, labeling the vertex $S_0$ as 0, let $\varphi(i)=i$ for $1\leq i \leq m-1$, $\varphi(0)=m$, $\varphi(m)=m+1$, $\varphi(m+1)=m+2$.
\end{proof}

Recall the following consequence of the AM-GM inequality (see for instance Theorem 2.1 of \cite{sidtrees}). Let $\mathbf{v}=(H_1, \ldots, H_s)$ where $H_i$'s are partially labelled graphs, and let $\mathbf{a}_1, \ldots, \mathbf{a}_k\in \NN^s$, $\alpha_1, \ldots, \alpha_k\in \RR_{> 0}$ and $\mathbf{b}=\sum_{i=1}^{k} \alpha_i \mathbf{a}_i$ such that $\mathbf{b}\in \NN^s$. Then $$[[\mathbf{v}^{\mathbf{a}_1}]]^{\alpha_1} \cdots [[\mathbf{v}^{\mathbf{a}_k}]]^{\alpha_k} \geq [[\mathbf{v}^{\mathbf{b}}]]$$ is a valid inequality where $\mathbf{v}^{\mathbf{a}_i}=H_1^{a_{i1}}\cdots H_s^{a_{is}}$.  We now prove a few inequalities using this observation. 

\begin{theorem}
    We have that $S_{2,1^0}^3 S_{2, 1^3}\geq S_{2, 1^2}^3$.
\end{theorem}
\begin{proof}
Applying AM-GM, we have that $$\left[\left[\Htwo{2}{1}{4}\right]\right]\left[\left[\left(\pthree{}{1}{2}{3}\right)^3\right]\right]^{\frac{1}{3}} \geq \left[\left[\Htwo{2}{1}{4} \pthree{}{1}{2}{3} \right]\right]$$ is a valid inequality, which proves the inequality above.
\end{proof}

\begin{theorem}
    We have that $S_{2,1^0} S_{2, 1^{m-2}}^2\geq  S_{2, 1^{m-1}}^2$.
\end{theorem}
\begin{proof}
Applying AM-GM, we have that 
$$\left[\left[\left(\vedge{1}{}\right)^2\right]\right]^{\frac{1}{2}} \left[\left[ \fivebroom{1}{2}{3}{\vdots}{}{m-1}{m}{m+1}\right]\right] \geq \left[\left[ \vedge{1}{} \fivebroom{1}{2}{3}{\vdots}{}{m-1}{m}{m+1} \right]\right]$$ 
is a valid inequality, which proves the inequality above.
\end{proof}

\begin{theorem}\label{thm:onetwoone}
    We have that $S_{2, 1^{i-2}} S_{2, 1^i}\geq S_{2, 1^{i-1}}^2$ for $i \geq 2$.
\end{theorem}
\begin{proof}
Applying AM-GM, we have that $$\left[ \left[ \left( \fourbroom{1}{2}{3}{\vdots}{i-1}{i}{i+1} \right)^2 \right] \right]^{\frac{1}{2}} \left[\left[ \left( \fivebroom{1}{2}{3}{\vdots}{i-1}{}{i}{i+1} \right)^2 \right]\right]^{\frac{1}{2}} \geq \left[\left[ \fourbroom{1}{2}{3}{\vdots}{i-1}{i}{i+1} \fivebroom{1}{2}{3}{\vdots}{i-1}{}{i}{i+1} \right]\right]$$ is a valid inequality, which proves the inequality above.
\end{proof}

Note that the inequality in Theorem~\ref{thm:onetwoone} could also have been seen as an application of H\"older's inequality to a sequence of products of certain powers of degrees of endpoints of edges. 

\begin{theorem}
    We have that $S_{2, 1^{m-2}}^{m}\geq S_{2, 1^{m-1}}^{m-1}$ for $m \geq 3$. Furthermore, when $m=2$, this is a valid inequality in paths.
\end{theorem}
\begin{proof}
When $m=2$, applying AM-GM, we have that $$\left[\left[ \Htwo{2}{1}{3} \right]\right] \left[\left[ \Htwo{3}{2}{4}\right]\right] \geq \left[\left[ \Htwo{2}{1}{3} \Htwo{3}{2}{4} \right]\right]$$ is a valid inequality, which yields $\uHtwo \ \uHtwo \geq \upthree$ as desired.

When $m\geq 3$, applying AM-GM, we have that $$ \left[\left[ \left( \fivebroom{3}{4}{5}{\vdots}{m}{m+1}{2}{1} \right)^{m-1}\right]\right]^{\frac{1}{m-1}} \left[\left[ \fivebroom{3}{4}{5}{\vdots}{m}{m+2}{2}{1}  \right]\right] \geq \left[\left[ \fivebroom{3}{4}{5}{\vdots}{m}{m+1}{2}{1} \fivebroom{3}{4}{5}{\vdots}{m}{m+2}{2}{1}  \  \right]\right] $$ is a valid inequality which is equivalent to $$ \left[\left[  \fivebroom{3}{4}{5}{\vdots}{m}{m+1}{2}{1} \right]\right] \left[\left[ \fivebroom{3}{4}{5}{\vdots}{m}{m+2}{2}{1}  \right]\right]^{m-1} \geq \left[\left[ \sixbroom{3}{4}{5}{}{m}{m+1}{m+2}{2}{1} \right]\right]^{m-1}$$ which proves the inequality above. \end{proof}

We now use Theorem~\ref{thm:kr11} to prove the validity of other important binomial inequalities for $\NU$.

\begin{theorem} We have that $\HDE(S_0 S_{2, 1^2}; S_{2, 1^0}) = 2$, and thus that $S_0 S_{2, 1^2} \geq S_{2, 1^0}^2$. \end{theorem}

\begin{proof}
    We first show that $\HDE(S_0 S_{2, 1^2}; S_{2, 1^0}) \leq 2$.  
    For every $\emptyset \neq S \subseteq V(S_{2,1^0})$, let $p(S) = 1$, and let $p(\emptyset)=0$. One can check that $p \in \cP(S_{2, 1^0})$.

    For any homomorphism $\varphi$ from $S_0 S_{2, 1^2}$ to $S_{2, 1^0}$ (including the optimal one), since the maximal cliques in $S_0 S_{2, 1^2}$ are the vertex in $S_0$ and the edges in $S_{2,1^2}$, we have that 
$$\sum_{S \subseteq \MC(S_0 S_{2, 1^2})} - (-1)^{|S|} p(\varphi(\cap S)) = 1 + 4 - 3 = 2.$$
Indeed, the vertex $S_0$ contributes $1$, the four edges of $S_{2,1^2}$ each contributes $1$, the vertex of degree two in $S_{2,1^2}$ contributes $-1$ as it is the intersection of two maximal cliques (edges), and the vertex of degree three in $S_{2,1^2}$ contributes $-2$ since it can be written in three different ways as the intersection of two maximal cliques, and in one way as the intersection of three maximal cliques. As we are minimizing over all $p\in \mathcal{P}(S_{2,1^0})$, this yields that $\HDE(S_0 S_{2, 1^2}; S_{2, 1^0}) \leq 2$ as desired.

We next show that $\HDE(S_0 S_{2, 1^2}; S_{2, 1^0}) \geq 2$. Consider the homomorphism $\varphi$ from $S_0 S_{2, 1^2}$ to $S_{2, 1^0}$ such that  $\varphi(0) = \varphi(4)= 3$, $\varphi(1) = \varphi(5) = 2$, and $\varphi(2) = \varphi(3) = 1$ (see Figure~\ref{fig:vert_sk2_12_p2_proof}). 

\begin{figure}\label{fig:vert_sk2_12_p2_proof}
    \centering
\includegraphics[scale=.5]{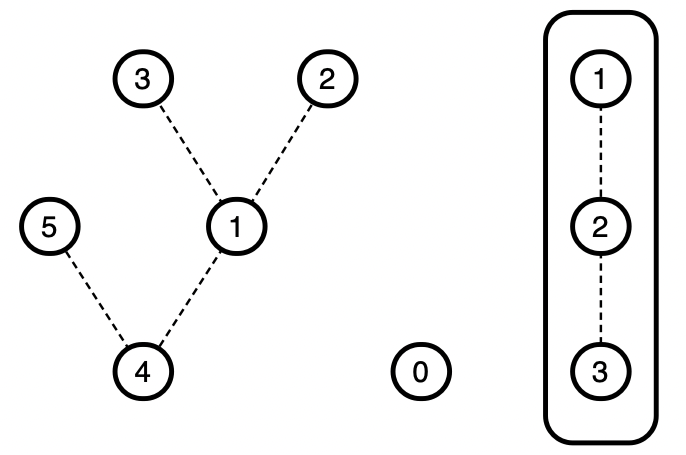}
    \caption{Illustration of $\varphi$.}
    \label{fig:vert_sk2_12_p2_proof}
\end{figure}
    
Every edge of $S_{2, 1^0}$ is covered by the image of an edge of $S_{2, 1^2}$ (each of which is a maximal clique thus contributing a weight of 1) twice. The inner vertex of $S_{2, 1^0}$ is covered by the image of the center vertex of $S_{2, 1^2}$, which yields a multiplier $-2$ (because it can be written as the intersection of three maximal cliques, and as three different intersections of two maximal cliques). One end vertex of $S_{2, 1^0}$ is not covered by the image of any inner vertices and leaves of $S_{2,1^2}$ contribute weights of 0 given that they cannot be written as intersections of maximal cliques. The other end vertex of $S_{2, 1^0}$ is covered by $S_0$ (which contributes a weight of $1$ since it is a maximal clique) and the inner vertex of the two-edge branch of $S_{2, 1^0}$ (which contributes a weight of $-1$ since it is the intersection of two maximal cliques), yielding a multiplier of 0. Thus we have
\begin{align*}
    \sum_{S \subseteq \MC(S_0 S_{2, 1^2})} &- (-1)^{|S|} p(\varphi(\cap S)) \\
    &= 2p(\{1,2\})+2p(\{2,3\})-2p(\{2\})\\
    &= 2p(\{1,2,3\})\\
    &=2,
\end{align*}
where the penultimate line follows from the last constraint of polymatroidal functions (letting $A=\{1,2\}$ and $B=\{2,3\}$, we have $p(\{2\})+p(\{1,2,3\})=p(\{1,2\})+p(\{2,3\})$ since $\{2\}$ separates $\{1\}$ and $\{3\}$), and the last line follows from the fact that we are considering normalized polymatroidal functions, meaning that $p(V(F_2))=1$.

Therefore for every $p \in \mathcal{P}(S_0 S_{2, 1^2})$, there exists a homomorphism that yields 2, and since we are maximizing over all homomorphisms, we have $\HDE(S_0, S_{2, 1^2}; S_{2, 1^0}) \geq 2$. This proves that $\HDE(S_0 S_{2, 1^2}; S_{2, 1^0}) = 2$.
 \end{proof}

\section{Rays of $\tropNU$}\label{sec:rays}

In this section, we show that different rays are in $\tropNU$. In the next section, we will see that the rays $\mathbf{d}_{j,m}$ are doubly extreme rays and the rays $\mathbf{s}_{j,m}$ and $\mathbf{r}_{j,m,i}$ are the remaining extreme rays of $\tropNU$. 

\begin{definition}
    Fix $m\in \NN$, and let
    \begin{itemize}
       \item $\mathbf{d}_{1,m}:=(1, 0, 0, \ldots, 0)$,
       \item $\mathbf{d}_{2,m}:=\vec{1}=(1,\ldots, 1)$,
       \item $\mathbf{d}_{3,m}:=(1,2,2,3,4,\ldots, m)$,
       \item $\mathbf{d}_{4,m} :=(1, 3, 4, 5, \ldots, m+2)$, and
       \item $\mathbf{d}_{5,m}:=(2, 4, 5, 7, \ldots, 2m+1)$
    \end{itemize}
    be vectors in $\RR^{m+1}$.
\end{definition}

\begin{lemma}\label{lem:constrdoubly}
The rays $\mathbf{d}_{j,m}$ are in $\tropNU$ for each $1 \leq j \leq 5$. 
\end{lemma}
\begin{proof}
We show how to realize each ray as $\alpha(\log \hom(S_0;G), \log \hom(S_{2,1^0};G), \ldots, \log \hom(S_{2,1^{m-1}};G))$ for some graph $G$ or as $\alpha(\log \hom(S_0;G_n), \log \hom(S_{2,1^0};G_n), \ldots, \log \hom(S_{2,1^{m-1}};G_n))$ for some sequence of graphs $G_n$ on $n$ vertices as $n\rightarrow \infty$ for some constant $\alpha\in \RR_{\geq 0}$.
\begin{enumerate}
\item Consider the graph $G_n$ consisting of $n-2$ isolated vertices and one edge. Then $\hom(S_0;G_n)=n$ and $\hom(S_{2,1^k};G_n)=2$ for $0\leq k \leq m-1$. As $n\rightarrow \infty$, we have $\frac{\log (\hom(S_0;G_n))}{\log (n)}=1$ and $\frac{\log(\hom(S_{2,1^k};G_n))}{\log(n)}\rightarrow 0$ for $0\leq k\leq m-1$, thus showing that $\mathbf{d}_{1,m}\in \tropNU$. 

\item To realize $\mathbf{d}_{2,m}=\vec{1}$, consider the graph $G$ consisting of a single edge. Note that $\hom(S_0; G)= \hom(S_{2,1^k}; G)  = 2$ for $0 \leq k \leq m-1$. Thus $\frac{\log(\hom(S_0; G))}{\log (2)} = \frac{\log(\hom(S_{2,1^k}; G))}{\log (2)} = 1$ for $0 \leq k \leq m-1$, and so $\mathbf{d}_{2,m}\in \tropNU$ as desired. 

\item To realize $\mathbf{d}_{3,m} = (1, 2, 2, 3, 4, \ldots, m)$, consider $S_n$, the star with $n$ branches. Then $\hom(S_0;S_n)=n+1$ and $\hom(S_{2,1^k};S_n)=n^{k+1}+n^2$ for $0\leq k \leq m-1$ (where the first term comes from sending the center of $S_{2,1^k}$ to the center of $S_n$, and the second term comes from sending it to one of the $n$ leaves of $S_n$). As $n\rightarrow \infty$, we have $\frac{\log (\hom(S_0;S_n))}{\log (n)}\rightarrow 1$, $\frac{\log(\hom(S_{2,1^0};S_n))}{\log(n)}\rightarrow 2$ and $\frac{\log(\hom(S_{2,1^k};S_n))}{\log(n)}\rightarrow k+1$ for $1\leq k\leq m-1$, thus showing that $\mathbf{d}_{3,m}\in \tropNU$.

\item To realize $\mathbf{d}_{4,m} =(1, 3, 4, 5, \ldots, m+2)$, consider the complete bipartite graph $K_{n,n}$. Then we have $\hom(S_0;K_{n,n})=2n$ and $\hom(S_{2,1^k};K_{n,n})=2n\cdot n^{k+1}\cdot n$ for $0\leq k \leq m-1$. As $n\rightarrow \infty$, we have $\frac{\log (\hom(S_0;K_{n,n}))}{\log (n)}\rightarrow 1$ and $\frac{\log(\hom(S_{2,1^k};K_{n,n}))}{\log(n)}\rightarrow k+3$ for $0\leq k\leq m-1$, thus showing that $\mathbf{d}_{4,m}\in \tropNU$.

\item To realize $\mathbf{d}_{5,m}=(2, 4, 5, 7, \ldots, 2m+1)$, let $G_{2n^2+1}$ be the tripartite graph with parts of size 1, $n^2$ and $n^2$, and where there are $n^2$ edges between the first and second part, $n^3$ edges between the second and third part, and no edges between the first and third part, i.e., this is the blow-up graph of a path of length two. Furthermore, the degree of every vertex in the second part is $n+1$ and the degree of every vertex in the third part is $n$. Then we have that $\hom(S_0;G_{2n^2+1})=2n^2+1$, $\hom(S_{2,1^0};G_{2n^2+1})=(n^2)^2+n^2\cdot (n+1)^2 + n^2\cdot n^2$, and $\hom(S_{2,1^k};G_{2n^2+1})=(n^2)^{k+1}(n+1)+n^2(n+1)^{k+1}(n^2+n)+n^2n^{k+1}(n+1)$ for $1\leq k \leq m-1$ where the different terms come from sending the center of $S_{2,1^k}$ to the different parts of $G_{2n^2+1}$ (see Figure~\ref{fig:ray-realize} for some examples). As $n\rightarrow \infty$, we have $\frac{\log (\hom(S_0;G_{2n^2+1}))}{\log (n)}\rightarrow 2$, $\frac{\log(\hom(S_{2,1^0};G_{2n^2+1}))}{\log(n)}\rightarrow 4$ and $\frac{\log(\hom(S_{2,1^k};G_{2n^2+1}))}{\log(n)}\rightarrow 2k+3$ for $1\leq k\leq m-1$, thus showing that $\mathbf{d}_{5,m}\in \tropNU$.

\end{enumerate}

\begin{figure}
    \centering
    \includegraphics[scale=.5]{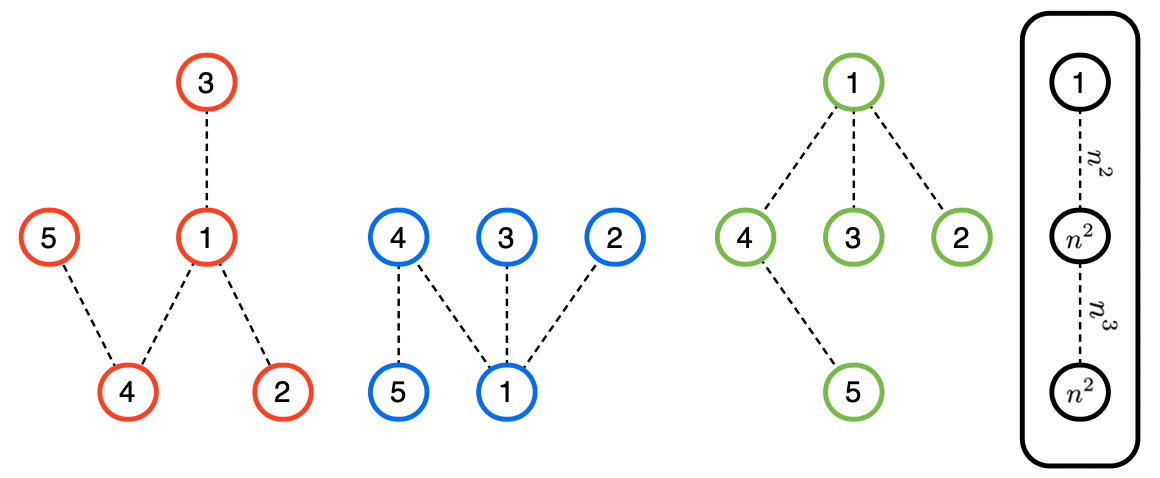}
    \caption{Some possible ways of sending $S_{2,1^2}$ to $G_{2n^2+1}$ for Case 5 of Lemma \ref{lem:constrdoubly}.}
    \label{fig:ray-realize}
\end{figure}
\end{proof}

We now define additional families of rays. In Theorem~\ref{thm:trop}, we will prove that these consist of the remaining extreme rays of $\tropNU$.

\begin{definition}
Fix $m\in \NN$ and let
\begin{itemize}
\item $\mathbf{s}_{1,m}:=(3,6,8,10,12,15,\ldots, 3(m-1), 3m)$, 

\item $\mathbf{s}_{2,m}:=(6,12,16,21,27,33,\ldots, 6m-3,6m+3)$, 

\item $\mathbf{s}_{3,m}:=(2,4,5,6,8,10, \ldots, 2m)$, 

\item $\mathbf{s}_{4,m} := (3,6,8,10,13,16, \ldots, 3m+1)$,

\item $\mathbf{r}_{1,m,i}:=(6i-15, 12i-30, 16i-40, 21i-54, 27i-72, 33i-90, \ldots, 6i^2-15i, 6i^2-9i-15, 6i^2-3i-30, \ldots, (6i-15)m)$ for $5\leq i\leq m-1$,  

\item $\mathbf{r}_{2,m,i}:=(i+2, 3i, 4i, 5i, \ldots, i(i+2), (i+1)(i+2), \ldots, (m-1)(i+2), m(i+2))$ for $5 \leq i \leq m-1$, 

\item $\mathbf{r}_{3,m,i}:=(3i-9,6i-18,8i-24,10i-30,13i-40, 16i-50, \ldots, 3i^2-12i+10, 3i^2-9i, 3i^2-6i-9, \ldots, (3i-9)m)$ for $5 \leq i \leq m-1$, 

\item $\mathbf{r}_{4,m,i}:=(2i-5,4i-10,5i-12,7i-18,9i-24, \ldots, 2i^2-7i+6, 2i^2-5i, 2i^2-3i-5, \ldots, m(2i-5))$ for $5 \leq i \leq m-1$, 

\item $\mathbf{r}_{5,m,i}:=(i, i, \ldots, i, i+1, i+2, \ldots, m)$ for $3 \leq i \leq m-1$, and 

\item $\mathbf{r}_{6,m,i}:=(i+1, 3i-1, 4i-2, 5i-3, \ldots, i^2+1, i^2+i, i^2+2i+1, \ldots, (m-1)(i+1), m(i+1))$ for $4\leq i \leq m-1$.
\end{itemize}
\end{definition}

For now, we only show that all of these rays are also in $\tropNU$.

\begin{lemma}\label{lem:extraysfromdoubly}
The rays $\mathbf{s}_{j,m}$ and $\mathbf{r}_{l,m,i}$ are in $\tropNU$ for $1\leq j \leq 4$ and $1 \leq l\leq 6$. 
\end{lemma}

\begin{proof}
Instead of finding explicit graphs to realize these rays as in the previous proof, we write them by taking tropical sums (i.e., by taking the maximum componentwise) of conical combinations of the rays $\mathbf{d}_{j,m}$ for $1\leq j \leq 5$. Since we have already shown that $\mathbf{d}_{j,m}\in \tropNU$ for $1\leq j \leq 5$ in Lemma~\ref{lem:constrdoubly} and since $\tropNU$ is a max-closed closed convex cone, this shows that these rays are also in $\tropNU$. Note that Section~\ref{sec:tropresults} explains how one could use this to come up with explicit graphs to realize these rays.

\begin{itemize}

    \item $\mathbf{s}_{1,m}=(3,6,8,10,12,15,\ldots, 3(m-1), 3m) = 3\bd_{3,m} \oplus 2\bd_{4,m}$. 

    \item $\mathbf{s}_{2,m}=(6,12,16,21,27,33,\ldots, 6m-3,6m+3) = 3\bd_{5,m} \oplus 4\bd_{4,m}$.

    \item $\mathbf{s}_{3,m}=(2,4,5,6,8,10, \ldots, 2m) = 2 \bd_{3,m} \oplus \bd_{5,m}$.

    \item $\mathbf{s}_{4,m} = (3,6,8,10,13,16, \ldots, 3m+1) = (\bd_{3,m} + \bd_{5,m}) \oplus 2 \bd_{4,m}$.

    \item $\mathbf{r}_{1,m,i}=(6i-15, 12i-30, 16i-40, 21i-54, 27i-72, 33i-90, \ldots, 6i^2-15i, 6i^2-9i-15, 6i^2-3i-30, \ldots, (6i-15)m)=(4i-10)\bd_{4,m} \oplus(6\bd_{4,m}+3(i-4)\bd_{5,m})\oplus (6i-15)\bd_{3,m}$ for $5\leq i\leq m-1$. 

    \item $\mathbf{r}_{2,m,i}=(i+2, 3i, 4i, 5i, \ldots, i(i+2), (i+1)(i+2), \ldots, (m-1)(i+2), m(i+2))= i\bd_{4,m} \oplus (i+2) \bd_{3,m}$ for $5 \leq i \leq m-1$.

    \item $\mathbf{r}_{3,m,i}=(3i-9,6i-18,8i-24,10i-30,13i-40, 16i-50, \ldots, 3i^2-12i+10, 3i^2-9i, 3i^2-6i-9, \ldots, (3i-9)m)= (\bd_{2,m}+(i-5)\bd_{3,m}+\bd_{4,m}+(i-3)\bd_{5,m})\oplus 2(i-3)\bd_{4,m} \oplus (3i - 9) \bd_{3,m}$ for $5 \leq i \leq m-1$. 

    \item $\mathbf{r}_{4,m,i}=(2i-5,4i-10,5i-12,7i-18,9i-24, \ldots, 2i^2-7i+6, 2i^2-5i, 2i^2-3i-5, \ldots, m(2i-5))=(2\bd_{4,m}+(i-4)\bd_{5,m})\oplus (2i-5)\bd_{3,m}$ for $5 \leq i \leq m-1$.

    \item $\mathbf{r}_{5,m,i}=(i, i, \ldots, i, i+1, i+2, \ldots, m)= i\bd_{2,m} \oplus \bd_{3,m}$ for $3 \leq i \leq m-1$. 

    \item $\mathbf{r}_{6,m,i}=(i+1, 3i-1, 4i-2, 5i-3, \ldots, i^2+1, i^2+i, i^2+2i+1, \ldots, (m-1)(i+1), m(i+1))= (2\bd_{2,m}+(i-1)\bd_{4,m})\oplus (i+1) \bd_{3,m}$ for $4\leq i \leq m-1$. 
    \end{itemize}
\end{proof}
\section{Main results}\label{sec:results}

In this section, we prove our main results. 

\maintheorem

\begin{proof}
By taking the $\log$ of the inequalities proven to be valid for $\NU$ in Section \ref{sec:inequalities}, we know that $\tropNU \subseteq Q$. 

To show that $Q \subseteq \tropNU$, we compute the extreme rays of $Q$ and show that they are all present in $\tropNU$. Any extreme ray must be tight with $m$ linearly independent inequalities and respect all other inequalities. Note that there are $m-2$ inequalities of type $y_{i-1}-2y_i+y_{i+1}\geq 0$ (which we call \emph{gap inequalities}) for $2 \leq i \leq m-1$. In particular, observe that we will never discuss the gap between $y_0$ and $y_1$. Moreover, there are seven other inequalities (which we call \emph{non-gap inequalities}) for a total of $m+5$ inequalities. So there could potentially be $\binom{m+5}{m}$ extreme rays. However, many of these intersections are not feasible. To better understand which of these $\binom{m+5}{m}$ intersections we need to consider, it is useful to first look at the projection of $\tropNU$ onto coordinates $y_0, y_1, y_2, y_3, y_4, y_{m-1}, y_m$ (i.e., the coordinates in the seven non-gap inequalities) which we call $C$. Computing $C$ is non-trivial for a general $m$, so we end up giving a larger set $C'$ that contains $C$ (as the defining inequalities of $C'$ are all valid for $C$) and which sufficiently restricts our search of extreme rays.

We first prove a few valid inequalities for $C$ which we use to define $C'$. Observe that $y_i \leq (i+2)y_0$ for all $i\in[m]$ since $S_{2,1^{i-1}}$ has $i+2$ vertices. Moreover, from the gap inequalities, we know the gap $y_{i+1}-y_i$ is greater or equal to the gap $y_i-y_{i-1}$ for $2\leq i \leq m-1$, i.e., gaps are non-decreasing, and so this implies that $y_{m-1}\geq y_3 + (m-4)(y_4-y_3)$, $y_m \geq y_3+(m-3)(y_4-y_3)$, and that $y_3 -y_4 - y_{m-1} + y_m \geq 0$.
Also, the conical combination of gap inequalities $\sum_{i=5}^{m-1} (i-4)(y_{i-1}-2y_i+y_{i+1}\geq 0)$ yields the inequality $y_4-(m-4)y_{m-1}+(m-5)y_m\geq 0$. 
Finally, the conical combination $(-y_1+y_2 \geq 0)+\sum_{i=2}^{j} (y_{i-1}-2y_i+y_{i+1}\geq 0)$ yields $-y_j+y_{j+1}\geq 0$ which implies that $y_1$ through $y_m$ are non-decreasing. 
Therefore, we let $$
C':= \left\{\yy \in \RR^7 \middle\vert \begin{array}{l}
- y_1 + y_2  \geq 0 \\
4y_1 - 3y_2 \geq 0 \\
3y_1 -3y_3 + y_4 \geq 0 \\
y_1 + 2y_{m-1} -2y_m \geq 0 \\
y_0 + y_{m-1} - y_m \geq 0 \\
y_0 - 2y_1 + y_3 \geq 0 \\
m\cdot y_{m-1} - (m-1)\cdot y_m \geq 0 \\
y_1 -2y_2+y_3\geq 0\\
y_2 -2y_3+y_4 \geq 0\\
3y_0-y_1 \geq 0\\
4y_0-y_2\geq 0\\
5y_0-y_3 \geq 0\\
6y_0-y_4 \geq 0\\
(m+1)y_0-y_{m-1} \geq 0\\
(m+2)y_0-y_m \geq 0\\
(m-5)y_3-(m-4)y_4+y_{m-1}\geq 0\\
(m-4)y_3-(m-3)y_4+y_m \geq 0\\
y_3 - y_4 - y_{m-1} + y_m \geq 0 \\
y_4 -(m-4)y_{m-1} + (m-5)y_m \geq 0 \\
-y_2 + y_3 \geq 0 \\
-y_3 + y_4 \geq 0 \\
-y_4 + y_{m-1} \geq 0 \\
-y_{m-1} + y_m \geq 0 
\end{array}\right\}\supseteq C.
$$
We can compute all the faces of $C'$ (for example in Sage or Polymake) to see which of the non-gap inequalities can be tight simultaneously. If the intersection of the hyperplanes corresponding to some of the non-gap inequalities with $C'$ yields $\mathbf{0}$, then we know their intersection with $C$ will also yield $\mathbf{0}$. Moreover, even if a particular set of $k$ of the first seven inequalities (our original non-gap inequalities) does not yield $\mathbf{0}$ in $C'$, we know that extreme rays in $C$ will require $m-k$ of the gap inequalities to be tight. In particular, if $k=4$ (respectively $k=3$ and $k=2$), then only two (respectively one and zero) gap inequalities aren't tight, and so at least one (respectively two and three) of $y_1 -2y_2+y_3\geq 0$, $y_2 -2y_3+y_4 \geq 0$ and $(m-5)y_3-(m-4)y_4+y_{m-1}\geq 0$ must be tight in order for there to potentially be an extreme ray in $C$ where these non-gap inequalities are tight. Taking all of this into consideration, Table~\ref{tab:tightineqs} gives the set of non-gap inequalities that can potentially be tight with extreme rays of $C$. For each set, we find which extreme rays can be recovered.

\begin{table}
\begin{tabular}{|l|c|c|c|c|c|c|c|}
\hline
&
\begin{turn}{90}$- y_1 + y_2  \geq 0$\end{turn}
&
\begin{turn}{90}$4y_1 - 3y_2 \geq 0$\end{turn}
&
\begin{turn}{90}$3y_1 -3y_3 + y_4 \geq 0$\end{turn}
&
\begin{turn}{90}$y_1 + 2y_{m-1} -2y_m \geq 0$\end{turn}
&
\begin{turn}{90}$y_0 + y_{m-1} - y_m \geq 0$\end{turn}
&
\begin{turn}{90}$y_0 - 2y_1 + y_3 \geq 0$\end{turn}
&
\begin{turn}{90}$m\cdot y_{m-1} - (m-1)\cdot y_m \geq 0$ \ \end{turn}\\
\hline
$\mathcal{S}_1$ & * & * & * & * & & &*\\
$\mathcal{S}_2$ & * & & & * & *& * &*\\
$\mathcal{S}_3$ & & * & * & * & * & & *\\
$\mathcal{S}_4$ & & * & * & * & * & &\\
$\mathcal{S}_5$ & & * & * & & * & * &\\
$\mathcal{S}_6$ & & * & * & & * & & *\\
$\mathcal{S}_7$ & & * & & * &* & &*\\
$\mathcal{S}_8$ & & & * &* &* & & *\\
$\mathcal{S}_9$ & & & & * & *& *& *\\
$\mathcal{S}_{10}$ & * & & & & & * & *\\
$\mathcal{S}_{11}$ & & * & & * & *& & \\
$\mathcal{S}_{12}$ & & & * & * & * & & \\
$\mathcal{S}_{13}$ & & & & & * & * & *\\
$\mathcal{S}_{14}$ & * & & & & & * & \\
\hline
\end{tabular}
    \caption{Cases of non-gap inequalities that can be tight simultaneously.}
    \label{tab:tightineqs}
\end{table}

$\mathcal{S}_1$: From the tight inequalities in this set, we have that $y_1=y_2=0$ and $y_{m-1}=y_m$. Since gaps are non-decreasing, this means that all gaps are 0, and so this yields the extreme ray $\mathbf{d}_{1,m}=(1,0,\ldots,0)$ which is tight with all $m-2$ gap inequalities. 

$\mathcal{S}_2$: From the tight inequalities in this set, we have that $y_0=a$, $y_1=y_2=2a$, $y_3=3a$, $y_{m-1}=(m-1)a$, $y_m=ma$ for some $a\geq 0$. Since the gap between $y_2$ and $y_3$ is the same as the gap as between $y_{m-1}$ and $y_m$ and gaps are non-decreasing, this yields the extreme ray $\mathbf{d}_{3,m}=(1,2,2,3,4,5,\ldots, m-1, m)$ which is tight with $m-3$ of the gap inequalities. 

$\mathcal{S}_3$: From the tight inequalities in this set, we have that $y_0=3a$, $y_1=6a$, $y_2=8a$, $y_{m-1}=3(m-1)a$, $y_m=3ma$ for some $a\geq 0$. We know there are at most three gap inequalities that are not additionally tight. Two cases arise: either $y_1-2y_2+y_3\geq 0$ is tight or it is not.

In the first case, we then have $y_3=10a$ which in turn implies that $y_4=12a$ since $3y_1 -3y_3 + y_4 = 0$. Note that the average gap between $y_4$ and $y_m$ has size $\frac{3ma-12a}{m-4}=3a$. Since the gap from $y_{m-1}$ to $y_m$ is $3a$ and gaps are non-decreasing, all gaps between $y_4$ and $y_m$ must have size $3a$. This yields the extreme ray $\mathbf{s}_{1,m}=(3,6,8,10,12,15,\ldots, 3(m-1), 3m)$ which is tight with $m-3$ of the gap inequalities.

In the second case when $y_1-2y_2+y_3> 0$, let $y_3=10a+b$ for some $b>0$. Then, since $3y_1 -3y_3 + y_4 = 0$, $y_4=12a+3b$, and so $y_2-2y_3+y_4>0$ as well. We know that gaps are non-decreasing, so the gap between $y_3$ and $y_4$ is at most the gap between $y_{m-1}$ and $y_m$ which means that $2b\leq a$. Note that we cannot have $2b=a$ since then every gap from $y_3$ to $y_m$ would be $3a$, and so $y_m=10a+b+(m-3)\cdot 3a\neq 3ma$. So $2b<a$ and $y_4-y_3<3a$, and since the gap between $y_{m-1}$ and $y_m$ is $3a$, there is one last gap inequality that isn't tight, say $y_{i-1} -2y_i + y_{i+1}>0$ for some $4\leq i \leq m-1$. So $y_i=10a+b+(i-3)(2a+2b)$ and $y_{i+1}=10a+b+(i-3)(2a+2b)+3a$, and we also know since all gaps thereafter are $3a$ and $y_m=3ma$, $y_{i+1}=(i+1)3a$. Thus, $b=\frac{(i-4)a}{2i-5}$. This yields the family of rays $\mathbf{r}_{1,m,i}=(6i-15, 12i-30, 16i-40, 21i-54, 27i-72, 33i-90, \ldots, 6i^2-15i, 6i^2-9i-15, 6i^2-3i-30, \ldots, (6i-15)m)$ which are tight with $m-5$ of the gap inequalities. For example, for $m=7$, we have $\mathbf{r}_{1,7,5}=(15, 30, 40, 51, 63, 75, 90,105) $ and $\mathbf{r}_{1,7,6}=(21, 42, 56, 72, 90, 108, 126, 147)$.  Finally, observe that we do not include $i=4$ in this family as this would make the inequality $y_1-2y_2+y_3=0$, and bring us back to $\mathbf{s}_{1,m}$.

$\mathcal{S}_4$: From the tight inequalities in this set, we have that $y_0=6a$, $y_1=12a$, $y_2=16a$, $y_3=c$, $y_4=3c-36a$, $y_{m-1}=2b$ and $y_m=6a+2b$ for some $a,b\geq 0$. Furthermore, we know that at most two gap inequalities aren't tight. Again, two cases arise: either $y_1-2y_2+y_3\geq 0$ is tight or it is not. 

In the first case, we get that $c=20a$, which implies that $y_4=24a$. The average gap between $y_4$ and $y_m$ must thus be $\frac{6a+2b-24a}{m-4}=\frac{2b-18a}{m-4}$. However, we know that $my_{m-1}-(m-1)y_m>0$, which is equivalent to $b>3a(m-1)$. Thus, the average gap is bigger than $6a$, which is a contradiction since the largest gap is $6a$. 

In the second case when $y_1-2y_2+y_3>0$, let $y_3=c=20a+d$ for some $d>0$. Then $y_4=24a+3d$. Further, note that $y_2-y_1=4a$, $y_3-y_2=4a+d$, $y_4-y_3=4a+2d$ (which are all distinct), and $y_m-y_{m-1}=6a$ and since at most two gap inequalities aren't tight, we must have that $4a+2d=6a$, and thus that $d=a$. We thus get the ray $\mathbf{s}_{2,m}=(6,12,16,21,27,33,\ldots, 6m-3,6m+3)$ which is tight with $m-4$ of the gap inequalities.

$\mathcal{S}_5$: From the tight inequalities in this set, we have that $y_0=b$, $y_1=3a$, $y_2=4a$, $y_3=6a-b$, $y_4=9a-3b$, $y_{m-1}=c$, and $y_m=b+c$ for some $a,b,c\geq 0$. Since gaps are non-decreasing, we have that $a\leq 2a-b\leq 3a-2b\leq b$ which implies that $a=b$ and that all gaps are $a$. This thus yields the ray $\mathbf{d}_{4,m}:=(1,3,4,5,6, \ldots, m+2)$ which is tight with all gap inequalities.

$\mathcal{S}_6$: From the tight inequalities in this set, we have that $y_0=b$, $y_1=3a$, $y_2=4a$, $y_3=c$, $y_4=3c-9a$, $y_{m-1}=(m-1)b$, $y_m=mb$ for some $a,b,c\geq 0$. Since gaps are non-decreasing, we have that $a \leq c-4a \leq 2c-9a \leq b$, but we know that at most two gap inequalities aren't tight, so not all of these can be distinct.

If $a = c-4a$, then $c=5a$ and $c-4a = 2c-9a$ as well. Suppose there are two gap inequalities that aren't tight, say $y_{i-1}-2y_i+y_{i+1}>0$ and $y_{j-1}-2y_j+y_{j+1}>0$ with $4\leq i < j \leq m-1$ (since we have assumed the inequalities for $i\in \{2,3\}$ are tight), then the $m-i+1$ variables $y_0, y_{i+1}, y_{i+2}, \ldots, y_m$ would be involved in $m-i-2$ tight gap inequalities and the tight inequality $y_0+y_{m-1}-y_m=0$, and therefore, this cannot result in a ray. So we know that in this case, we can assume that there is at most one gap inequality that isn't tight, say $y_{i-1}-2y_i+y_{i+1}>0$ for some $4\leq i\leq m-1$. Then $y_i=(i+2)a=ib$ which implies that $b=\frac{i+2}{i}a$. This yields the family of rays $\mathbf{r}_{2,m,i}=(i+2, 3i, 4i, 5i, \ldots, i(i+2), (i+1)(i+2), \ldots, (m-1)(i+2), m(i+2))$ which are tight with $m-3$ of the gap inequalities. For example, for $m=7$, we have $\mathbf{r}_{2,7,5}=(7,15,20,25,30,35,42,49)$, $\mathbf{r}_{2,7,6}=(8,18,24,30,36,42,48,56)$. Finally, observe that we do not include $i=4$ in this family as this would make the inequality $y_1+2y_{m-1}-2y_m \geq 0$ tight.

Now let's consider the case when $a<c-4a$, say $c=5a+d$, which in turn implies that $c-4a < 2c-9a$, and since at most two gap inequalities aren't tight, we must have that $2c-9a = b=a+2d$. We thus have that $y_m=m(a+2d)=y_3+(m-3)(a+2d)=5a+d+(m-3)(a+2d)$, and so $a=\frac{5d}{2}$. In turn, this gives that $y_1=\frac{15d}{2}$, $y_m=m\frac{9d}{2}$  and $y_{m-1}=(m-1)\frac{9d}{2}$, and therefore, $y_1+2y_{m-1}-2y_m<0$, a contradiction.

$\mathcal{S}_7$: From the tight inequalities in this set, we have that $y_0=3a$, $y_1=6a$, $y_2=8a$, $y_3=b$, $y_4=c$, $y_{m-1}=3(m-1)a$ and $y_m=3ma$ for some $a,b,c\geq 0$. Suppose that $y_1-2y_2+2y_3=0$. In that case, $b=10a$ and since $3y_1-3y_3+y_4>0$, let $c=12a+d$ for some $d>0$. We know there are at most two gap inequalities that aren't tight, and the three gaps present are $2a\leq 2a+d \leq 3a$ which are all distinct since $d>0$ and since if $2a+d=3a$, then $y_m$ would be equal to $10a+(m-3)3a \neq 3ma$. We already know that $y_2-2y_3+y_4>0$. 

Suppose $y_{i-1}-2y_i+y_{i+1}>0$ as well for some $5\leq i\leq m-1$. Then $y_i=3ia=10a+(i-3)(2a+d)$, which implies that $d=\frac{i-4}{i-3}a$. This thus yields the family of rays $\mathbf{r}_{3,m,i}=(3i-9,6i-18,8i-24,10i-30,13i-40, 16i-50, \ldots, 3i^2-12i+10, 3i^2-9i, 3i^2-6i-9, \ldots, (3i-9)m)$ which are tight with $m-4$ gap inequalities. For example, $\mathbf{r}_{3,7,5}=(6,12,16,20,25,30,36,42)$ and $\mathbf{r}_{3,7,6}=(9, 18, 24, 30, 38, 46, 54, 63)$. Finally, observe that we do not include $i=4$ in this family as this would make the inequality $3y_1-3y_3+y_4\geq 0$ tight. 

Now suppose that $y_1-2y_2+2y_3>0$ and let $y_3=b=10a+d$. Since $3y_1-3y_3+y_4>0$, let $y_4=c=12a+3d+f$ for some $f>0$. There are thus four gaps present, namely $2a<2a+d<2a+2d+f\leq 3a$ and since at most two gap inequalities are not tight, we must have $2a+2d+f=3a$, which implies that $y_4=y_m-(m-4)3a=12a\neq 12a+3d+f$, a contradiction.

$\mathcal{S}_8$: From the tight inequalities in this set, we have that $y_0=a$, $y_1=2a$, $y_2=c$, $y_3=b$, $y_4=3b-6a$, $y_{m-1}=(m-1)a$, $y_m=ma$ for some $a,b,c\geq 0$.

Suppose that $y_1-2y_2+y_3=0$ and let $y_2-y_1=d$ for some $d> 0$ (since $y_2-y_1>0$), which implies that $y_2=2a+d$, $y_3=2a+2d$, $y_4=6d$. Since $4y_1-3y_2>0$, we have that $d<\frac{2a}{3}$. Moreover, since gaps are non-decreasing, we know that $y_4-y_3\geq d$, which implies that $d\geq \frac{2a}{3}$, a contradiction.

Now suppose that $y_1-2y_2+y_3>0$, and let $y_2=2a+d$ and $y_3=2a+2d+f$ for some $d,f>0$, which implies that $y_4=6d+3f$. There are thus four gaps present, namely $d<d+f\leq 4d+2f-2a \leq a$, and since at most two gap inequalities are not tight, we need to have either that $4d+2f-2a = a$ or $d+f= 4d+2f-2a$.

If $4d+2f-2a = a$, then $y_m=ma=2a+2d+f+a(m-3)$, which implies that $6a=12d+6f$, contradicting $4d+2f-2a = a$.

If $d+f= 4d+2f-2a$ and the change from gaps of size $d+f=4d+2f-2a$ to gaps of size $a$ happens at position $4 \leq i\leq m-1$, i.e., $y_{i-1}-2y_i+y_{i+1} >0$, then $y_i=(i+2)d+(i-1)f=i\frac{3d+f}{2}$, which implies that $f=\frac{i-4}{i-2}d$. This thus yields the family of rays $\mathbf{r}_{4,m,i}=(2i-5,4i-10,5i-12,7i-18,9i-24, \ldots, 2i^2-7i+6, 2i^2-5i, 2i^2-3i-5, \ldots, m(2i-5))$ which are tight with $m-4$ gap inequalities. For example, $\mathbf{r}_{4,7,5}=(5, 10, 13, 17, 21, 25, 30, 35)$ and $\mathbf{r}_{4,7,6}=(7, 14, 18, 24, 30, 36, 42, 49)$.  Finally, observe that we do not include $i=4$ in this family as this would make the inequality $4y_1-3y_2\geq 0$ tight.

$\mathcal{S}_9$: From the tight inequalities in this set, we have that $y_0=a$, $y_1=2a$, $y_2=b$, $y_3=3a$, $y_4=4a$, $y_{m-1}=(m-1)a$, $y_m=ma$ for some $a,b\geq 0$. Since we know that gaps are non-decreasing, we know that $y_i=ia$ for $3\leq i\leq m$. 

If $y_1-2y_2+y_3=0$, then $y_2=b=\frac{5}{2}a$, and so the only gap inequality that isn't tight is $y_2-2y_3+y_4\geq 0$. This yields the ray $\mathbf{s}_{3,m}=(2,4,5,6,8,10, \ldots, 2m)$.

If $y_2-2y_3+y_4=0$, then $y_2=b=2a$, but then $-y_1+y_2=0$, a contradiction.

Finally, note that we can't have an extreme ray for which both $y_1-2y_2+y_3>0$ and $y_2-2y_3+y_4>0$ since then no tight inequality would involve $y_2$.

$\mathcal{S}_{10}$: From the tight inequalities in this set, we have that $y_0=b$, $y_1=a$, $y_2=a$, $y_3=2a-b$, $y_{m-1}=(m-1)c$, $y_m=mc$ for some $a,b,c\geq 0$. Note that there is at most one gap inequality that isn't tight, and there are three gaps already present: $0\leq a-b \leq c$, so at least two of those must be equal.

If $a-b=0$ and the gap switches from $0$ to $c$ at position $2\leq i \leq m-1$, i.e., $y_{i-1}-2y_i+y_{i+1}>0$, this yields the family of rays $\mathbf{r}_{5,m,i}=(i, i, \ldots, i, i+1, i+2, \ldots, m)$. Finally, observe that we do not include $i=2$ in this family as this would make the inequality $y_1+2y_{m-1}-2y_m\geq 0$ tight.

If $a-b=c$, then $y_2=a=2c$ since each gap after $y_2$ is $c$, but then $y_1+2y_{m-1}-2y_m=0$, a contradiction.

$\mathcal{S}_{11}$: From the tight inequalities in this set, we have that $y_0=\frac{3a}{2}$, $y_1=3a$, $y_2=4a$, $y_3=c$, $y_{m-1}=b-\frac{3a}{2}$, $y_m=b$ for some $a,b\geq 0$. We know that there is at most one gap inequality that isn't tight, say $y_{i-1}-2y_i+y_{i+1}>0$.

If $i\geq 4$, then $y_3=5a$ and $y_4=6a$, which means that $3y_1-3y_3+y_4=0$, a contradiction. 

If $i=3$, this yields the ray $\mathbf{s}_{4,m}=(3,6,8,10,13,16, \ldots, 3m+1)$.

If $i=2$, this implies that $y_3=\frac{11a}{2}$ and $y_4=7a$ which implies that $3y_1-3y_3+y_4<0$, a contradiction.

$\mathcal{S}_{12}$: From the tight inequalities in this set, we have that $y_0=a$, $y_1=2a$, $y_3=c$, $y_4=3c-6a$, $y_{m-1}=b-a$, $y_m=b$ for some $a,b,c\geq 0$. We know that there is at most one gap inequality that isn't tight, so the gap between $y_3$ and $y_4$, $2c-6a$, must be either equal to the gap between $y_{m-1}$ and $y_m$, namely $a$, or to the gap between $y_2$ and $y_3$.

If $y_4-y_3=a$, then $y_4=c+a=3c-6a$, which implies that $a=\frac{2c}{7}$. Moreover, either $y_1-2y_2+y_3=0$ or $y_2-2y_3+y_4=0$. In the first case, this implies that $y_2=\frac{11c}{14}$, but then $4y_1-3y_2<0$, a contradiction. In the second case, this implies that $y_2=\frac{5c}{7}$ and yields the extreme ray $\mathbf{d}_{5,m}=(2,4,5,7,\ldots, 2m-1, 2m+1)$. 

If $y_4-y_3=y_3-y_2$ but $y_4-y_3<a$, then we must have that $y_2-y_1=y_3-y_2$ as well since there is at most one gap inequality that isn't tight. So $y_0=a$, $y_1=2a$, $y_2=2a+d$, $y_3=2a+2d$, $y_4=2a+3d$ for some $d \geq 0$. Since $3y_1-3y_3+y_4=0$, we have that $2a=3d$, but then $4y_1-3y_2=0$, a contradiction.

$\mathcal{S}_{13}$: From the tight inequalities in this set, we have that $y_0=a$, $y_1=b$, $y_3=2b-a$, $y_{m-1}=(m-1)a$ and $y_m=ma$ for some $a,b\geq 0$. We know that at most one gap inequality isn't tight with any extreme ray in this setting, so $y_1-2y_2+y_3=0$ or $y_2-2y_3+y_4=0$.

If $y_1-2y_2+y_3=0$, $y_2=\frac{3b-a}{2}$. We already have two gaps present: $\frac{b-a}{2}<a$ (note that they cannot be equal since $4y_1-3y_2>0$. Suppose that the switch from one gap to the other happens at position $i$, i.e., $y_{i-1}-2y_i+y_{i+1}>0$ for some $3\leq i \leq m-1$. Since $y_i=2ia=(i+1)b-(i-1)a$, we have that $b=\frac{(3i-1)a}{i+1}$ which yields the family of rays $\mathbf{r}_{6,m,i}=(i+1, 3i-1, 4i-2, 5i-3, \ldots, i^2+1, i^2+i, i^2+2i+1, \ldots, (m-1)(i+1), m(i+1))$. For example we have $\mathbf{r}_{6,7,6}=(7,17,22,27,32,37,42,49)$, $\mathbf{r}_{6,7,5}=(6,14,18,22,26,30,36,42)$ and $\mathbf{r}_{6,7,4}=(5,11,14,17,20,25,30,35)$. Note that we exclude $i=3$ since in that case $y_1+2y_{m-1}-2y_m=0$, a contradiction.

If $y_1-2y_2+y_3>0$, then $y_2-2y_3+y_4=0$ and the gap between $y_2$ and $y_3$ must be equal to the gap between $y_{m-1}$ and $y_m$, namely $a$, which implies that we have that $y_i=ia$ for $2\leq i\leq m$. Moreover, we also have $y_2=2b-2a$ and $y_4=2b$. So $b=2a$, but then $-y_1+y_2=0$, a contradiction. 

$\mathcal{S}_{14}$: From the tight inequalities in this set, we have that $y_1=a$ and $y_2=a$ for some $a\geq 0$. Moreover, we know that all gap inequalities must be tight, so $y_i=a$ for $1\leq i \leq m$. Finally, since $y_0-2y_1+y_3=0$, we have that $y_0=a$. This yields the ray $\mathbf{d}_{2,m}=(1,1,\ldots, 1)$.

Since all the extreme rays produced for $Q$ were shown to be in $\tropNU$ by Lemma~\ref{lem:constrdoubly} and Lemma~\ref{lem:extraysfromdoubly}, we have that $Q\subseteq \tropNU$ as desired.
\end{proof}

\maincorollary

\begin{proof}
This follows from Lemma~\ref{lem:extraysfromdoubly} and Theorem~\ref{thm:trop}, and the fact that one can check that none of the doubly extreme rays is in the double hull of the others.
\end{proof}
\section{Final remarks}\label{sec:further}

We first note that although we did not extensively use the results of Kopparty and Rossman \cite{koppartyrossman} in how we presented our results, they were immensely useful in coming up with the tropicalization. We also remark that our results show similarities to the tropicalization of stars described in \cite[Theorem 2.17]{BR}.

\begin{theorem}
Let $\mathcal{U}=\{S_0, S_1, \ldots, S_m\}$ where $S_i$ is the star graph with $i$ branches. Then $$\tropNU=\left\{ \begin{array}{lll}
\mathbf{y}\in\mathbb{R}^{m+1} | & -y_{1} + y_{2} \geq 0 &\\
& y_0 + y_{m-1} - y_m \geq 0 & \\
& y_{i-1}-2y_{i}+y_{i+1} \geq 0 & \forall 1 \leq i \leq m-1\\
& m \cdot y_{m-1} - (m-1) \cdot y_{m} \geq 0 & 
\end{array} \right\}.$$ 

The set $\tropNU$ is the double hull of the following doubly extreme rays: $(1,0,0,\ldots, 0)$, $(1,1,\ldots, 1)$, $(1,1,2,3,\ldots, m)$, and $(1,2,3,4,\ldots,m+1)$. 
\end{theorem}

Furthermore, in this paper, we proved Theorem~\ref{thm:trop} by considering all extreme rays. However, it is possible to prove this theorem using only the doubly extreme rays (without finding all extreme rays) by doing some additional casework. This proof will be included in the thesis of the first-named author. 

Finally, we are interested in computing tropicalizations of other families of trees. One family of particular interest is a generalization of our $S_{2,1^k}$'s that includes star-like graphs where branches can have any length. We think of these almost-stars through the partition of their number of edges according to their branches. By restricting which kinds of partitions are allowed, it might be possible to compute other tropicalizations. Note that the tropicalization of paths in \cite{BR} is much more complicated than the tropicalization of stars. Therefore, given that these almost-stars contain longer paths, computing their tropicalization might be highly non-trivial.
\bibliographystyle{alpha}
\bibliography{ref}
\end{document}